\documentclass[12pt,a4paper]{article}

\usepackage[ngerman,english]{babel}
\usepackage[utf8]{inputenc}
\usepackage[T1]{fontenc}
\usepackage{amsmath}
\usepackage{amssymb}
\usepackage{amsthm}
\usepackage{mathtools}
\usepackage{paralist}
\usepackage{tikz}
\usepackage{graphicx}
\usepackage{hyperref}
\usepackage[capitalise]{cleveref}

\begin{document}

\theoremstyle{plain}
\newtheorem{theorem}{Theorem}[section]
\newtheorem{definition}[theorem]{Definition}
\newtheorem{lemma}[theorem]{Lemma}
\newtheorem{prop}[theorem]{Proposition}
\newtheorem{corollary}[theorem]{Corollary}
\newtheorem{conjecture}[theorem]{Conjecture}
\theoremstyle{remark}
\newtheorem{remark}[theorem]{Remark}
\newtheorem{example}[theorem]{Example}

\newcommand{\reg}{\mathrm{reg}}
\newcommand{\Ass}{\mathrm{Ass}}
\newcommand{\charakt}{\mathrm{char}}
\newcommand{\diag}{\mathrm{diag}}
\newcommand{\Tor}{\mathrm{Tor}}
\newcommand{\im}{\mathrm{im}}
\newcommand{\coker}{\mathrm{coker}}
\newcommand{\id}{\mathrm{id}}
\newcommand{\length}{\mathrm{length}}
\newcommand{\LM}{\mathrm{LM}}
\newcommand{\LT}{\mathrm{LT}}
\newcommand{\cone}{\mathrm{cone}}
\newcommand{\ord}{\mathrm{ord}}
\newcommand{\Quot}{\mathrm{Quot}}
\newcommand{\Spec}{\mathrm{Spec}}
\newcommand{\height}{\mathrm{ht}}
\newcommand{\rank}{\mathrm{rank}}
\newcommand{\Ann}{\mathrm{Ann}}
\newcommand{\reynolds}{\mathcal{R}}
\newcommand{\maxId}{\mathfrak{m}}
\newcommand{\maxIdn}{\mathfrak{n}}
\newcommand{\primId}{\mathfrak{p}}

\title{A degree bound for rings of arithmetic invariants}
\author{David Mundelius \\ \small{Technische Universität München, Zentrum Mathematik - M11} \\ \small{Boltzmannstraße 3, 85748 Garching, Germany} \\ \small{\texttt{david.mundelius@tum.de}}}
\date{May 27, 2022}
\maketitle

\begin{abstract}
Consider a Noetherian domain $R$ and a finite group $G \subseteq Gl_n(R)$. We prove that if the ring of invariants $R[x_1, \ldots, x_n]^G$ is a Cohen-Macaulay ring, then it is generated as an $R$-algebra by elements of degree at most $\max(|G|,n(|G|-1))$. As an intermediate result we also show that if $R$ is a Noetherian local ring with infinite residue field then such a ring of invariants of a finite group $G$ over $R$ contains a homogeneous system of parameters consisting of elements of degree at most $|G|$.
\end{abstract}




\section*{Introduction}

A celebrated theorem of Symonds \cite{symondsdb1,symondsdb2} states that if $K$ is an arbitrary field and $G \subseteq Gl_n(K)$ is a finite subgroup, then the ring of invariants $K[x_1, \ldots, x_n]^G$ is generated as a $K$-algebra by elements of degree at most $\max(|G|,n(|G|-1))$. This result had been proved earlier in unpublished work of Abraham Broer under the additional assumption that $K[x_1, \ldots, x_n]^G$ is a Cohen-Macaulay ring, see \cite[Theorem 3.9.8]{dked1}. The main result of this article is a generalization of Broers result to the situation where the field $K$ is replaced by an arbitrary Noetherian integral domain. Some results regarding the question when rings of invariants over $\mathbb{Z}$ are Cohen-Macaulay rings can be found in \cite{almuhaimeedcm}.

In this article all rings are assumed to be commutative, all graded rings are assumed to be $\mathbb{N}$-graded, and by a graded $R$-algebra for some ring $R$ we mean a graded ring $S=\bigoplus_{i \in \mathbb{N}} S_i$ with $S_0 \cong R$. For a ring $R$ and a subgroup $G \subseteq Gl_n(R)$ we always consider the action of $G$ on $R[x_1, \ldots, x_n]$ by $\sigma(f)=f(\sigma^{-1}(x_1, \ldots, x_n))$ for $f \in R[x_1, \ldots, x_n]$ and $\sigma \in G$.

In \cref{SectionReduct} some reduction results are given, which will later allow us to restrict ourselves to the case where $R$ is a Noetherian local domain with infinite residue field. Under this assumption we prove in \cref{SectionParaSys} that the ring of invariants always contains a homogeneous system of parameters which consists of elements of degree at most $|G|$. All results of these first two sections hold without the assumption that the ring of invariants is a Cohen-Macaulay ring.

In order to prove the main result we then show that, if $f_1, \ldots, f_n$ is such a system of parameters and $A \coloneqq R[f_1, \ldots, f_n]$, then under the given assumptions the ring of invariants is generated as an $A$-module by elements of degree at most $n \cdot (|G|-1)$. As in the proof of Symonds' theorem this is done by showing that the Castelnuovo-Mumford regularity of the ring of invariants is at most zero. In \cref{SectionLocCohom} we study the local cohomology modules involed in the definition of Castelnuovo-Mumford regularity; for this part the assumption that the ring of invariants is a Cohen-Macaulay ring is essential. Finally, in \cref{SectionMain} we use this to prove the aforementioned bound on the Castelnuovo-Mumford regularity and then derive the main result from that.

\section*{Acknowledgement}

I wish to thank Gregor Kemper for many helpful conversations.

\section{Reductions} \label{SectionReduct}

The following basic lemma will be used several times within this article:

\begin{lemma} \label{InvRingFlat}
Let $R$ be a ring, $R'$ a flat $R$-algebra, and $G \subseteq Gl_n(R)$ a finite subgroup. Then $R'[x_1, \ldots, x_n]^G=R[x_1, \ldots, x_n]^G \otimes_R R'$.
\end{lemma}

\begin{proof}
We write $S \coloneqq R[x_1, \ldots, x_n]$ and $S_{R'}=R'[x_1, \ldots, x_n]$. There is an exact sequence of $R$-modules
\[ 0 \to S^G \to S \stackrel{\varphi}{\to} \bigoplus_{\sigma \in G} S \]
with $\varphi(f)=(\sigma(f)-f)_{\sigma \in G}$ for all $f \in S$. By tensoring this sequence with $R'$ we obtain an exact sequence
\[ 0 \to S^G \otimes_R R' \to S_{R'} \stackrel{\varphi_{R'}}{\to} \bigoplus_{\sigma \in G} S_{R'} \]
where $\varphi_{R'}(f)=(\sigma(f)-f)_{\sigma \in G}$ for all $f \in S_{R'}$. This implies that $S^G \otimes_R R' \cong \ker \varphi_{R'}=S^G_{R'}$.
\end{proof}

For a Noetherian ring $R$ and a finite group $G \subseteq Gl_n(R)$ we define $\beta_R(G)$ to be the smallest integer $k$ such that $R[x_1, \ldots, x_n]^G$ is generated as an $R$-algebra by elements of degree at most $d$. Our first application of \cref{InvRingFlat} is the following result which shows that in the proof of the main theorem we may always replace $R$ by some faithfully flat $R$-algebra:

\begin{lemma} \label{ReductFaithful}
Let $R$ be a ring, $R'$ a faithfully flat $R$-algebra and $G \subseteq Gl_n(R)$ a finite subgroup. Then $\beta_{R'}(G)=\beta_R(G)$.
\end{lemma}

\begin{proof}
Set again $S \coloneqq R[x_1, \ldots, x_n]$ and $S_{R'} \coloneqq R'[x_1, \ldots, x_n]$. Then by \cref{InvRingFlat} we have $S^G_{R'} = S^G \otimes_R R'$, so $\beta_{R'}(G) \leq \beta_R(G)$. Assume that $\beta_{R'}(G) < \beta_R(G)$ and let $B$ be the subalgebra of $S^G$ generated by all elements of degree at most $d \coloneqq \beta_R(G)-1$. Then $B \otimes_R R'$ is a subalgebra of $S_{R'}^G$ which contains all elements of degree at most $d \geq \beta_{R'}(G)$, so by assumption it is $S_{R'}^G$ itself. Therefore we obtain that $B \otimes_R R'=S_{R'}^G=S^G \otimes_R R'$ and therefore $B=S^G$ since $R'$ is faithfully flat. This contradicts the definition of $B$, so we must have $\beta_{R'}(G)=\beta_R(G)$.
\end{proof}

The next goal is to reduce the main theorem to the case where $R$ is local. For this, we first need the following graded version of Nakayama's Lemma:

\begin{lemma} \label{GradedNakayama}
Let $R$ be a ring, $S$ a finitely generated graded $R$-algebra, and $M=\bigoplus_{i \in \mathbb{N}} M_i$ a nonnegatively graded $S$-module. Let moreover $U \subseteq M$ be a set of homogeneous elements. Then $U$ generates $M$ as an $S$-module if and only if it generates $M/S_+M$ as an $R$-module.
\end{lemma}

For the proof of this we refer to \cite[Lemma 3.7.1]{dk}; there it is assumed that $R$ is a field, but this assumption is nowhere used in the proof.

Now we can prove the desired reduction of the main theorem to the case where $R$ is local.

\begin{lemma} \label{ReductLocal}
Let $R$ be a ring and let $G \subseteq Gl_n(R)$ be a finite subgroup. Then 
\[ \beta_R(G) \leq \max \{ \beta_{R_\maxId}(G) | \maxId \subset R \text{ is a maximal ideal} \}. \]
\end{lemma}

\begin{proof}
Again set $S \coloneqq R[x_1, \ldots, x_n]$ and $S_{R_\maxId} \coloneqq R_{\maxId}[x_1, \ldots, x_n]$ for every maximal ideal $\maxId \subset R$. Let $B$ be the subalgebra of $S^G$ generated by all elements of degree at most $\max \{ \beta_{R_\maxId}(G) | \maxId \subset R \text{ is a maximal ideal} \}$. By \cref{InvRingFlat} we then have $B \otimes_R R_\maxId=S_{R_\maxId}^G$ for each maximal ideal $\maxId \subset R$. With $M \coloneqq S^G/B_+ S^G$ we have $M \otimes_R R_\maxId \cong S^G_{R_\maxId}/(S^G_{R_\maxId})_+ = R_\maxId$; more precisely, if we view $M$ as a graded module $M=\bigoplus_{i \in \mathbb{N}} M_i$, then $M_i \otimes_R R_\maxId=0$ for all $i>0$. Since this holds for every maximal ideal, we have $M_i=0$ for all $i>0$ and therefore $M=M_0=R$. Now we can apply \cref{GradedNakayama} to see that $S^G$ is generated by elements of degree $0$ as an $B$-module, and therefore $S^G=B$.
\end{proof}

\section{Homogeneous systems of parameters} \label{SectionParaSys}

Let $R$ be a ring and let $S$ be a finitely generated graded $R$-algebra. A sequence $f_1, \ldots, f_n$ of homogeneous elements in $S$ is called a homogeneous system of parameters if $f_1, \ldots, f_n$ are algebraically independent over $R$ and $S$ is finitely generated as a module over $A \coloneqq R[f_1, \ldots, f_n]$. In general, a finitely generated graded algebra does not contain a homogeneous system of parameters, see \cite{be} and the references there. In this section we prove that, if $R$ is a Noetherian local ring, the ring of invariants of every finite subgroup $G \subseteq Gl_n(R)$ contains a homogeneous system of parameters. Moreover, if in addition the residue field of $R$ is infinite, this system of parameters can be chosen to be consisting of elements of degree at most $|G|$. We start with two technical lemmas:

\begin{lemma} \label{ParaSysLemma1}
Let $R$ be a local ring with maximal ideal $\maxId$ and $F \coloneqq R/\maxId$. Moreover, let $S$ be a graded $R$-algebra and $M=\bigoplus_{i \in \mathbb{Z}} M_i$ a graded $S$-module such that each $M_i$ is finitely generated as an $R$-Module. Then for every sequence $g_1, \ldots, g_m$ of homogeneous elements in $M$ the following holds: if the classes $\overline{g}_1, \ldots, \overline{g}_m$ of $g_1, \ldots, g_m$ in $M/\maxId M$ generate $M/\maxId M$ as an $F$-vector space, then $g_1, \ldots, g_m$ generate $M$ as an $R$-module.
\end{lemma}

\begin{proof}
Let $N$ be the $R$-module generated by $g_1, \ldots, g_m$. For $i \in \mathbb{Z}$ we write $N_i \coloneqq M_i \cap N$. By assumption we have $M/\maxId M=N/\maxId N$ and therefore $M_i/\maxId M_i=N_i/\maxId N_i$ for each $i \in \mathbb{Z}$. Each $N_i$ is again a finitely generated $R$-module, generated by some of the elements $g_1, \ldots, g_m$. The classes of these generators then generate $M_i/\maxId M_i=N_i/\maxId N_i$ as an $F$-vector space. Since $M_i$ is finitely generated as an $R$-module, Nakayama's lemma now implies that $M_i=N_i$. Since this holds for every $i$, we obtain $M=N$.
\end{proof}

\begin{lemma} \label{ParaSysLemma2}
Let $R$ be a local ring with maximal ideal $\maxId$ and $F \coloneqq R/\maxId$. Let $S$ be a finitely generated graded $R$-algebra and let $f_1, \ldots, f_n \in S$ be homogeneous elements and $A \coloneqq R[f_1, \ldots, f_n]$. Moreover, set $S_F \coloneqq S \otimes_R F$ and $A_F \coloneqq A \otimes_R F=F[\overline{f}_1, \ldots, \overline{f}_n]$ where $\overline{f}_i$ denotes the class of $f_i$ over $F$. If $S_F$ is finitely generated as an $A_F$-module, then $S$ is finitely generated as an $A$-module.
\end{lemma}

\begin{proof}
By \cref{GradedNakayama} it is sufficient to prove that $M \coloneqq S/A_+S$ is a finitely generated $R$-module. Since $S$ is a finitely generated graded $R$-algebra, for every $i \in \mathbb{N}$ the degree-$i$-part of $S$ is a finitely generated $R$-module and therefore the same holds for the degree-$i$-part of $M$. We have $M/ \maxId M= S/(A_+ \cup \maxId)S=S_F/(A_F)_+$ and since $S_F$ is a finitely generated $A_F$-module, $S_F/(A_F)_+$ is a finitely generated $F \cong A_F/(A_F)_+$-vector space, so $M/\maxId M$ is a finitely generated $R$-module. Therefore \cref{ParaSysLemma1} implies that $M$ is indeed a finitely generated $R$-module.
\end{proof}

Now we are ready to prove the following result, which gives a condition when a subset of a ring of invariants over a local ring is a homogeneous system of parameters:

\begin{theorem} \label{ParaSysThm}
Let $R$ be a Noetherian local ring with maximal ideal $\maxId$ and $G \subseteq Gl_n(R)$ a finite group; set $F \coloneqq R / \maxId$, $S \coloneqq R[x_1, \ldots, x_n]$, and $S_F \coloneqq F[x_1, \ldots, x_n]=S \otimes_R F$. Moreover, let $f_1, \ldots, f_n$ be a sequence of homogeneous elements of $S^G$. If the classes $\overline{f}_1, \ldots, \overline{f}_n$ of $f_1, \ldots, f_n$ in $S_F$ form a homogeneous system of parameters in $S^G_F$, then $f_1, \ldots, f_n$ form a homogeneous system of parameters in $S^G$.
\end{theorem}

\begin{proof}
Let $A \coloneqq R[f_1, \ldots, f_n]$ and $A_F \coloneqq A \otimes_R F = F[\overline{f}_1, \ldots, \overline{f}_n]$. By assumption $S^G_F$ is a finitely generated $A_F$-module. It is well-known that $S_F$ is integral over $S^G_F$, so $S_F$ is also a finitely generated $A_F$-module. By \cref{ParaSysLemma2} this implies that $S$ is a finitely generated $A$-module. Since $A$ is Noetherian, $S^G$ is also a finitely generated $A$-module.
\end{proof}

Note that in the preceding proof \cref{ParaSysLemma2} cannot be applied directly to $S^G$ since it is in general not true that $S^G_F = S^G \otimes_R F$.

The following result also appeared in \cite[Corollary 7.38]{dissMundelius}, but using \cref{ParaSysThm} we can give a much more elementary proof for it:

\begin{corollary} \label{ExistParaSys}
Let $R$ be a noetherian local ring with maximal ideal $\maxId$ and set $F \coloneqq R/\maxId$ and $S\coloneqq R[x_1, \ldots, x_n]$. Let $G \subseteq Gl_n(R)$ be a finite group. Then $S^G$ contains a homogeneous system of parameters.
\end{corollary}

\begin{proof}
Let $g_1, \ldots, g_n \in S$ be homogeneous elements such that their classes $\overline{g}_1, \ldots, \overline{g}_n \in S_F \coloneqq F[x_1, \ldots, x_n]$ are invariants which form a homogeneous system of parameters in $S_F^G$. This is possible since every finitely generated graded algebra over a field contains a homogeneous system of parameters, see e.g. \cite[Corollary 2.5.8]{dk}. For each $j=1, \ldots, n$ we set $f_j \coloneqq \prod_{\sigma \in G} \sigma(g_j) \in S^G$. Since $\overline{g}_j$ is already invariant, the classes of $f_1, \ldots, f_n$ in $S_F^G$ are $\overline{g}_1^{|G|}, \ldots, \overline{g}_n^{|G|}$, which also form a homogeneous system of parameters. Now it follows from \cref{ParaSysThm} that $f_1, \ldots, f_n$ form a homogeneous system of parameters in $S^G$.
\end{proof}

For the proof of our main theorem we need a bound on the degrees of the elements of the system of parameters; this is possible with one additional assumption:

\begin{corollary} \label{ParaSysDegBound}
Let $R$ be a noetherian local ring with maximal ideal $\maxId$ and set $F \coloneqq R/\maxId$  and $S \coloneqq R[x_1, \ldots, x_n]$. Let $G \subseteq Gl_n(R)$ be a finite group. If $F$ has infinitely many elements, then $S^G$ contains a homogeneous system of parameters consisting of elements of degree at most $|G|$.
\end{corollary}

\begin{proof}
Since $F$ is infinite, we can choose homogeneous elements $g_1, \ldots, g_n \in S$ of degree one such that for each $j$ the class $\overline{g}_j$ of $g_j$ in $S_F=F[x_1, \ldots, x_n]$ is not contained in the $F$-vector space generated by all $\sigma(\overline{g}_k)$ with $\sigma \in G$ and $1 \leq k < j$. Furthermore we set $f_j \coloneqq \prod_{\sigma \in G} \sigma(g_j)$. Then $f_j \in S^G$ for each $j$ and a classical result of Dade, see e.g. \cite[Proposition 3.5.2]{dk}, shows that the classes of $g_1, \ldots, g_n$ in $S_F$ form a homogeneous system of parameters in $S_F^G$. Now each $f_j$ is homogeneous of degree $|G|$, and \cref{ParaSysThm} shows that these elements form a system of parameters.
\end{proof}

\section{Some local cohomology modules} \label{SectionLocCohom}

We start this section with a basic lemma:

\begin{lemma} \label{GradedExtensionRadicals}
Let $R$ be a ring and let $A \subseteq B$ be an integral extension of graded $R$-algebras. Then for every $a \in R$ we have the following equality of ideals in $B$: $\sqrt{(A_+ +(a)) B}=\sqrt{B_+ +(a)}$.
\end{lemma}

\begin{proof}
It is sufficient to show that every homogeneous element $f \in B_+$ is in $\sqrt{A_+ B}$. Let $d \coloneqq \deg(f)>0$. Since $A \subseteq B$ is integral we have elements $a_0, \ldots, a_{n-1} \in A$ such that $f^n=a_{n-1} f^{n-1} + \ldots a_1 f +a_0$. This equality remains valid if each $a_i$ is replaced by its homogeneous part of degree $d \cdot (n-i)$. Then in particular each $a_i$ is in $A_+$, so we get $f^n \in A_+ B$ and therefore $f \in \sqrt{A_+ B}$.
\end{proof}

Note that the element $a$ does not play any essential role in this proof; we need the lemma with arbitrary $a$ below, so it is given in this generality.

For the rest of this section, we fix the following notation: let $R$ be a Noetherian local ring with an element $a \in R$ which is neither a unit nor a zero divisor; write $R_a \coloneqq R[\frac{1}{a}]$. Let moreover $G \subseteq Gl_n(R)$ be a finite group. We set $S \coloneqq R[x_1, \ldots, x_n]$ and $S_a \coloneqq R_a[x_1, \ldots, x_n]$. Finally, let $A \subseteq S^G$ be an $R$-subalgebra of $S^G$ generated by a homogeneous system of parameters. The goal of this section is to study the local cohomology modules $H^i_{A_+}(S^G)$ in the case where $S^G$ is a Cohen-Macaulay ring. We start with an auxiliary result:

\begin{lemma} \label{LocCohomLemma}
With the notation as above we have homogeneous isomorphisms of graded local cohomology modules $H^i_{A_+ +(a)}(S^G) \cong H^i_{S^G_+ +(a)}(S^G)$ and $H^i_{A_+}(S^G[\frac{1}{a}]) \cong H^i_{(S^G_a)_+}(S^G_a)$ for all $i \in \mathbb{N}$.
\end{lemma}

\begin{proof}
By the Graded Independence Theorem for local cohomology (see \cite[Theorem 14.1.7]{bs}) we have a homogeneous isomorphism 
\[ H^i_{A_+ +(a)}(S^G) \cong H^i_{(A_+ +(a))S^G}(S^G) =H^i_{\sqrt{(A_+ +(a))S^G}}(S^G). \] By \cref{GradedExtensionRadicals} we have $\sqrt{(A_+ +(a))S^G}=\sqrt{S^G_+ +(a)}$, so the first claimed isomorphism follows.

For the second isomorphism we first note that by \cref{InvRingFlat} we have $S^G[\frac{1}{a}]=(S_a)^G$. Using the Graded Independence Theorem and \cref{GradedExtensionRadicals} we obtain as above:
\[ H^i_{A_+}(S_a^G) \cong H^i_{A_+ S^G}(S_a^G)=H^i_{\sqrt{A_+ S^G}}(S_a^G)=H^i_{S^G_+}(S_a^G). \]
Now using the Graded Independence Theorem again we obtain
\[ H^i_{S^G_+}(S_a^G)=H^i_{ (S^G[\frac{1}{a}])_+}(S_a^G)=H^i_{(S_a^G)_+}(S_a^G). \]
By putting everything together, the second claim follows.
\end{proof}

With this lemma we can prove some properties of the local cohomology modules $H^i_{A_+}(S^G)$ which we will need in the next section:

\begin{theorem} \label{LocCohom}
With the notation introduced before \cref{LocCohomLemma} assume in addition that $S^G$ is a Cohen-Macaulay ring. Then $H^i_{A_+}(S^G)=0$ for all $i \neq n$ and we have a homogeneous injective map $H^n_{A_+}(S^G) \to H^n_{(S^G_a)_+}(S^G_a)$.
\end{theorem}

\begin{proof}
By \cite[Exercise 14.1.11]{bs} we have an exact sequence of graded $A$-modules
\begin{align*}
 0 &\to H^0_{A_+ +(a)}(S^G) \to H^0_{A_+}(S^G) \to H^0_{A_+}(S^G[\tfrac{1}{a}]) \\
&\to H^1_{A_+ +(a)}(S^G) \to H^1_{A_+}(S^G) \to H^1_{A_+}(S^G[\tfrac{1}{a}]) \\
&\to \ldots \\
&\to H^i_{A_+ +(a)}(S^G) \to H^i_{A_+}(S^G) \to H^i_{A_+}(S^G[\tfrac{1}{a}]) \\
&\to \ldots
\end{align*}
Using \cref{LocCohomLemma} we can rewrite this sequence as
\begin{align*}
 0 &\to H^0_{S^G_+ +(a)}(S^G) \to H^0_{A_+}(S^G) \to H^0_{(S_a^G)_+}(S_a^G) \\
&\to H^1_{S^G_+ +(a)}(S^G) \to H^1_{A_+}(S^G) \to H^1_{(S_a^G)_+}(S_a^G) \\
&\to \ldots \\
&\to H^i_{S^G_+ +(a)}(S^G) \to H^i_{A_+}(S^G) \to H^i_{(S_a^G)_+}(S_a^G) \\
&\to \ldots
\end{align*}
Since $S^G$ is a graded Cohen-Macaulay ring we have $H^i_{S^G_+ +(a)}(S^G)=0$ for all $i<\height(S^G_++(a))=n+1$ by \cite[Theorem 6.2.7]{bs}. Since $S_a^G$ is also a Cohen-Macaulay ring we get in the same way $H^i_{(S_a^G)_+}(S_a^G)=0$ for $i<n$. With the above exact sequence this yields the claims for $i \leq n$. Moreover, since the ideal $A_+$ is generated by $n$ elements, we have $H^i_{A_+}(S^G)=0$ for all $i>n$ by \cite[Theorem 3.3.1]{bs}.
\end{proof}

\section{Castelnuovo-Mumford regularity and the main result} \label{SectionMain}

Let $A$ be a graded ring and let $M=\bigoplus_{i \in \mathbb{Z}} M_i$ be a finitely generated graded $A$-module. Then one defines $\mathrm{end}(M) \coloneqq \sup \{i \in \mathbb{Z} | M_i \neq 0 \}$. Moreover, one defines the Castelnuovo-Mumford regularity of $M$ as $\reg(A,M) \coloneqq \sup_{j \in \mathbb{N}} (\mathrm{end}(H^j_{A_+}(M))+j)$. 

\begin{remark} \label{RegSubalgebra} \label{RegLocaliz} \
\begin{compactenum}[(i)]
\item If $B \subseteq A$ is a graded subalgebra such that $A$ is finitely generated as a $B$-module, then by \cref{GradedExtensionRadicals} and the Graded Independence Theorem (see \cite[Theorem 14.1.7]{bs}) we get $\reg(A,M)=\reg(B,M)$ for every finitely generated graded $A$-module $M$.
\item For a maximal ideal $\maxId \subset A_0$ we write $M_\maxId \coloneqq M \otimes_{A_0} (A_0)_\maxId$. Then clearly $\mathrm{end}(M)$ is the supremum over all $\mathrm{end}(M_\maxId)$ where $\maxId$ ranges over all maximal ideals in $A_0$. Using the Graded Flat Base Change Theorem for local cohomology (see \cite[Theorem 14.1.9]{bs}) we get that $\mathrm{end}(H^j_{A_+}(M))$ is the supremum over all $\mathrm{end}(H^j_{(A_\maxId)_+}(M_\maxId))$ and therefore $\reg(A,M)$ is the supremum over all $\reg(A_\maxId,M_\maxId)$.
\end{compactenum}
\end{remark}

\begin{theorem} \label{RegDegBound}
Let $R$ be a Noetherian ring and let $A$ be a finitely generated graded $R$-algebra which is generated by homogeneous elements $f_1, \ldots, f_n \in A_+$. Moreover, let $M$ be a finitely generated  nonnegatively graded $A$-module. Then $M$ is generated as an $A$-module by elements of degree at most $\reg(A,M)+\sum_{i=1}^n (\deg(f_i)-1)$.
\end{theorem}

This is proved in the case where $R$ is a field in \cite[Proposition 2.1]{symondsdb1} and in the case that all $f_i$ are of degree $1$ in \cite[Theorem 16.3.1]{bs}. The proof given here is similar to the one in \cite{bs}.

\begin{proof}
Let $a \in A_+$ be a nonzero homogeneous element and $d \coloneqq \deg(a)$. Then we have an exact sequence of graded $A$-modules
\[ 0 \to M(-d) \stackrel{\cdot a}{\to} M \to M/aM \to 0. \]
By \cite[Exercise 16.2.15(iv) and Remark 14.1.10(ii)]{bs} we obtain
\begin{align*} \reg(A,M/aM) &\leq \max(\reg(A,M(-d))-1,\reg(A,M)) \\ &=\max(\reg(A,M)+d-1,\reg(A,M))=\reg(A,M)+d-1. \end{align*}
Using this repeatedly, we find 
\[ \reg(A,M/A_+M) \leq \reg(A,M)+\sum_{i=1}^n (\deg(f_i)-1). \]
Since $M/A_+M$ is an $A_+$-torsion module, we have $M/A_+M \cong H^0_{A_+}(M/A_+M)$ and hence $\mathrm{end}(M/A_+M) =\mathrm{end}(H^0_{A_+}(M/A_+M)) \leq \reg(A,M/A_+M)$. In particular, $M/A_+M$ is generated by elements of degree at most $\mathrm{end}(M/A_+M) \leq \reg(A,M/A_+M) \leq \reg(A,M)+\sum_{i=1}^n (\deg(f_i)-1)$. Now the theorem follows from \cref{GradedNakayama}.
\end{proof}

The next result is essentially a rephrasing of \cref{LocCohom} in terms of Castelnuovo-Mumford regularity:

\begin{prop} \label{BoundReg1}
Let $R$ be a Noetherian local ring with an element $a \in R$ which is neither a unit nor a zero divisor and set $R_a \coloneqq R[\frac{1}{a}]$. Let moreover $G \subseteq Gl_n(R)$ be a finite group. We set $S \coloneqq R[x_1, \ldots, x_n]$ and $S_a \coloneqq R_a[x_1, \ldots, x_n]$. Assume that $S^G$ is a Cohen-Macaulay ring. Then $\reg(S^G,S^G) \leq \reg(S_a^G,S_a^G)$.
\end{prop}

\begin{proof}
Let $A \subseteq S^G$ be an $R$-subalgebra generated by a homogeneous system of parameters. Then by \cref{RegSubalgebra}(i) we have $\reg(S^G,S^G)=\reg(A,S^G)$. \cref{LocCohom} shows that 
\[ \reg(A,S^G)=\mathrm{end}(H^n_{A_+}(S^G))-n \leq \mathrm{end}(H^n_{(S_a^G)_+}(S_a^G))-n \leq \reg(S_a^G,S_a^G). \]
\end{proof}

Using this we can prove a bound on $\reg(S^G,S^G)$:

\begin{prop} \label{BoundReg2}
Let $R$ be a Noetherian integral domain. Let moreover $G \subseteq Gl_n(R)$ be a finite group. We set $S \coloneqq R[x_1, \ldots, x_n]$. Assume that $S^G$ is a Cohen-Macaulay ring. Then $\reg(S^G,S^G) \leq 0$.
\end{prop}

\begin{proof}
Using \cref{RegLocaliz}(ii) we may reduce this to the case where $R$ is local and therefore $\dim(R)<\infty$. However, in the following we allow $R$ to be a not necessarily local ring of finite Krull dimension, because otherwise the following induction argument would not work. Namely we use induction on the dimension of $R$: in the case $\dim(R)=0$, $R$ itself must be a field; then the proposition is proved in \cite{symondsdb1}. If $\dim(R)>0$, by \cref{RegLocaliz}(ii) we may restrict ourselves to the case that $R$ is local. Choose an element $0 \neq a \in R$ which is not a unit. Then $\dim(R_a)<\dim(R)$ since $a$ must be contained in the unique maximal ideal of $R$, so the claim follows from Proposition \ref{BoundReg1} and the induction hypothesis; note that $R_a$ need not be local, so it is essential that we made the reduction to the local case only within the induction step.
\end{proof}

Recall that $\beta_R(G)$ denotes the smallest integer $k$ such that $R[x_1, \ldots, x_n]^G$ is generated by elements of degree at most $k$ as an $R$-algebra. We are now ready to prove a first bound on $\beta_R(G)$.

\begin{theorem} \label{MainParaSys}
Let $R$ be a Noetherian integral domain, $G \subseteq Gl_n(R)$ a finite group and $S \coloneqq R[x_1, \ldots, x_n]$. Assume that $S^G$ is a Cohen-Macaulay ring which contains a homogeneous system of parameters $f_1, \ldots, f_n$. Then  $S^G$ is generated as a module over $R[f_1, \ldots, f_n]$ by elements of degree at most $\sum_{i=1}^n (\deg(f_i)-1))$. Moreover,
\[ \beta_R(G) \leq \max(\deg(f_1), \ldots, \deg(f_n), \sum_{i=1}^n (\deg(f_i)-1)). \]
\end{theorem}

\begin{proof}
The first claim follows from \cref{RegDegBound} and Proposition \ref{BoundReg2}. The second claim is then clear since elements which generate $S^G$ as an $R[f_1, \ldots, f_n]$-module together with $f_1, \ldots, f_n$ generate $S^G$ as an $R$-algebra.
\end{proof}

\begin{remark}
If the ring $R$ in \cref{MainParaSys} is regular local, then one can give a much simpler proof of the theorem. By \cite[§4, no. 3, Corollaire]{bourbakicax} and \cite[Proposition 1.5.15(d)]{bh} it then follows that $S^G$ is a free module over $A=R[f_1, \ldots, f_n]$. Since with $K=\Quot(R)$, $S_K=K[x_1, \ldots, x_n]$, and $A_K=K[f_1, \ldots, f_n]$ we have $S_K^G=S^G \otimes_R A_K$, it follows that a minimal generating set of $S^G$ as an $A$-module consists of elements of the same degrees as a minimal generating set of $S_K^G$ as an $A_K$-module. But over fields the theorem is well-known, see \cite{symondsdb1}. However, this does not imply that $\beta_R(G)=\beta_K(G)$ as there may be a homogeneous system of parameters for theb invariant ring over $K$ which consists of elements of smaller degrees than a homogeneous system of parameters for the invariant ring over $R$.
\end{remark}

Using this bound we can now deduce the main result:

\begin{theorem}
Let $R$ be a Noetherian integral domain, $G \subseteq Gl_n(R)$ a finite group and $S \coloneqq R[x_1, \ldots, x_n]$. If $S^G$ is a Cohen-Macaulay ring, then $\beta_R(G) \leq \max(|G|,n(|G|-1))$.
\end{theorem}

\begin{proof}
Note that for any maximal ideal $\maxId \subset R$ the elements $f_1, \ldots, f_n$ also form a system of parameters for the invariant ring $R_\maxId[x_1, \ldots, x_n]^G$. By \cref{ReductLocal} it is therefore sufficient to consider the case where $R$ is local with maximal ideal $\maxId$. Set $R' \coloneqq R[x]_{\maxId R[x]}$. This is a faithfully flat local $R$-algebra with an infinite residue field (see \cite[Example 16.2.4]{bs}), so by \cref{ReductFaithful} we can restrict ourselves to the case where $R$ itself has an infinite residue field.

Then \cref{ParaSysDegBound} shows that there is a system of parameters $f_1, \ldots, f_n \in S^G$ with $\deg(f_i) \leq |G|$ for each $i$. Therefore we can deduce the claim from \cref{MainParaSys}.
\end{proof}

\bibliographystyle{plainurl}
\bibliography{AInvDegBound}

\end{document}